\documentclass[12pt]{article}

\usepackage{amsfonts}
\usepackage{graphics}
\usepackage{amssymb}

\newtheorem{theorem}{Theorem}[section]
\newtheorem{lemma}{Lemma}[section]
\newtheorem{definition}{Definition}[section]
\newtheorem{proposition}[theorem]{Proposition}
\newtheorem{corollary}[theorem]{Corollary}
\newtheorem{example}{Example}[section]
\newtheorem{remark}{Remark}[section]

\begin{document}

\begin{center}
{\Large Tempered Generalized Functions and Hermite Expansions\\}

\end{center}

\vspace{0.3cm}

\begin{center}
{\large Pedro Catuogno\footnote{Research partially supported by CNPQ 302704/2008-6.} and
Christian Olivera\footnote{Research partially supported by FAPESP 02/10246-2, 2007/54740-4.}}\\

\textit{Departamento de
 Matem\'{a}tica, Universidade Estadual de Campinas, \\ F. 54(19) 3521-5921 – Fax 54(19)
3521-6094\\ 13.081-970 -
 Campinas - SP, Brazil. e-mail: pedrojc@ime.unicamp.br ; colivera@ime.unicamp.br}
\end{center}
\vspace{0.3cm}
\begin{center}
\begin{abstract}
\noindent In this work we introduce a new algebra of tempered
generalized functions. The tempered distributions are embedded in
this algebra via their Hermite expansions. The Fourier transform is
naturally extended to this algebra in such a way that the usual
relations involving multiplication, convolution and differentiation
are valid. We study the elementary properties of the association,
embedding, point values and Fourier transform. Furthermore, we give
a generalized It\^o formula in this context and some applications to
stochastic analysis.
\end{abstract}
\end{center}


\vspace{0.3cm} \noindent {\bf MSC2000 subject classification:}
 46F30, 60H.

\section {Introduction}

\noindent The differential algebras of generalized functions of
Colombeau type were developed in connection with non linear
problems. These algebras provide a good framework to solve
differential equations with rough initial data or discontinuous
coefficients (see \cite{Biag}, \cite{gkos} and \cite{ober2}).
Recently there has been a great interest in developing a stochastic
calculus in algebras of generalized functions (see \cite{albe},
\cite{capa}, \cite{mar1}, \cite{mi-se}, \cite{ober4} and
\cite{Russo} and references).

\noindent  In order to develop a Fourier transform for generalized
functions, J. F. Colombeau in \cite{colb} introduced an algebra of
tempered generalized functions with the characteristic that most of
the properties involving the Fourier transform and convolution are
valid only in a weak sense. For a recent account of the theory and
applications we refer to  the reader to \cite{Del} and \cite{gkos}.

\noindent The following question naturally arises: is it possible to
construct an algebra of generalized functions in which the
operations of multiplication, convolution, differentiation and
Fourier transform are defined and have all the properties of the
ordinary operations?

\noindent In this work we introduced a new algebra of tempered
generalized functions that answers the above question. The
construction of this algebra is based on the Fourier-Hermite
expansion of tempered distributions. More precisely, the Hermite
representation theorem for $\mathcal{S}^{\prime}$ (see \cite{Miku},
\cite{Red1}, \cite{Schw}, and \cite{Red3}) establishes that every
$S\in\mathcal{S}^{\prime}$ can be represented by a Hermite series
\begin{equation}\label{repre}
S=\sum_{j=0}^{\infty} S(h_j)h_j
\end{equation}
where $\{h_j\}$ are the Hermite functions and the equality is in the
weak sense. The idea is to embed the tempered distributions into a
differential subalgebra $\mathcal{H}_{\mathbf{s}^{\prime}} \subset
\mathcal{S}^{\mathbb{N}_0}$ via the sequence of partial sums
\[
S_{n}=\sum_{j=0}^{n} S(h_j)h_j
\]
and define the algebra of tempered generalized functions as
\[
\mathcal{H}=\mathcal{H}_{\mathbf{s}^{\prime}}/\mathcal{H}_{\mathbf{s}}
\]
where $\mathcal{H}_{\mathbf{s}}$ is a differentiable ideal of
$\mathcal{H}_{\mathbf{s}^{\prime}}$ (see Section 3 for precise
definitions and details).

\noindent The operations in $\mathcal{H}$ are naturally introduced
in a pointwise way. By example the Fourier transform of $[f_n] \in
\mathcal{H}$ is defined to be $[\mathcal{F}(f_n)]$, where
$\mathcal{F}$ is the usual Fourier transform. Now, it is not
difficult to prove that the Fourier transform is an isomorphism
and that usual identities involving convolution, derivatives and
Fourier transform are satisfied.

\noindent Another important point of this paper is that we obtain an
It\^o formula under weak regularity conditions. We make use of
F\"{o}llmer's formulation of pathwise stochastic calculus \cite{Fol}
in order to establish an It\^o formula for tempered  generalized
functions. This formula extends the classical one via association
and the embedding of the tempered distributions into $\mathcal{H}$.

\noindent The paper is organized as follows: Section 2 contains
a brief summary without proofs of Hermite functions and the
Hermite representation theorems. In section 3, we introduce the
tempered algebra $\mathcal{H}$, this algebra contains the
tempered distributions and extends the product in $\mathcal{S}$.
Moreover, we study some elementary properties and we show that the
symmetric Hermite product of tempered distributions (see
\cite{calo} and \cite{Shen-Sun}) is associated with the product in
$\mathcal{H}$.

\noindent In section 4, in order to define the point value of
tempered generalized functions, we introduce the ring of tempered
numbers $\mathbf{h}$. Section 5 we deal with integration,
convolutions and the Fourier transform of tempered generalized
functions. We obtain the Fourier inversion theorem for
$\mathcal{H}$, the formula of interchange between the product and
the convolution and the rule of integration by parts. We remark that
the identities are in $\mathbf{h}$, this is in a strong sense.

\noindent In section 6 we introduce an elementary stochastic
calculus for generalized tempered functions. We derive an It\^o
formula for elements of $\mathcal{H}$. We observe that the
classical It\^o formula is associated to the generalized It\^o
formula, under suitable regularity conditions. Moreover, we obtain
a probabilistic representation of the solution of the heat
equation with data in $\mathcal{H}$. Finally we note that in
\cite{mar2} and \cite{mar1} is present an It\^o formula for
generalized functions containing some mistakes in definitions and
proofs; see \cite{capa}.

\section{N-Representation of tempered distributions.}
Let $\mathcal{S} \equiv \mathcal{S}(\mathbb{R})$ be the Schwartz
space of rapidly decreasing smooth real valued functions.

For each $m\in \mathbb{N}_0 \equiv \mathbb{N}\cup\{0\}$, we
consider $\|\cdot\|_{m}$ the norm of $\mathcal{S}$ given by
\[
\|\varphi\|_{m}=\Big(\int^{\infty}_{-\infty}|
(N+1)^m\varphi(x)\mid^2dx \Big)^{\frac{1}{2}},
\]
where $N+1=\frac{1}{2}(-\frac{d^2}{dx^2}+x^2+1)$.

We observe that $\mathcal{S}$ is provided with the natural topology
given by these norms is a sequentially complete locally convex
space and its dual space $\mathcal{S}^{\prime}$ is the space of
tempered distributions. The family of norms $\{ \|\cdot\|_{m}: m
\in \mathbb{N}_0\}$ is direct and equivalent to the family of
seminorms $\{ \|\cdot\|_{\alpha, \beta, \infty}: \alpha, \beta \in
\mathbb{N}_0\}$, given by
\[
\| \varphi \|_{\alpha, \beta, \infty}=\sup_x |(1+
|x|^{2})^{\alpha} D^{\beta} \varphi(x)|.
\]
We use often the following property of the multiplication on
$\mathcal{S}$. For all $m \in \mathbb{N}_0$ there exists $r, s \in
\mathbb{N}_0$ and a constant $C_m>0$ such that
\begin{equation}\label{eq1}
\|\varphi \psi \|_m \leq C_m \| \varphi \|_r \| \psi \|_s
\end{equation}
for all $\varphi, \psi \in \mathcal{S}$ (see for instance
\cite{radyno}, Theorem 2).

The {\it{Hermite polynomials}} $H_n(x)$ are defined by

\begin{equation}\label{forpolHe}
    H_n(x)=(-1)^n e^{\frac{x^2}{2}} \frac{d^n}{dx^n}  e^{-\frac{x^2}{2}}
\end{equation}
for $n\in \mathbb{N}_{0}$ or equivalently
\begin{equation}\label{forpolHp}
      H_n(x) =2^{-\frac{n}{2}}\sum_{k=0}^{[n/2]}\frac{(-1)^k
      n!( \sqrt{2}x)^{n-2k}}{k!(n-2k)!}.
\end{equation}
The {\it{Hermite functions}} $h_n(x)$ are defined by
\begin{equation}\label{relfunpolh}
    h_{n}(x)=  (\sqrt{2\pi}n!)^{-\frac{1}{2}}
    e^{{-\frac{1}{4}x^2}}H_n(x)
\end{equation}
for $n\in \mathbb{N}_{0}$. Some properties of the Hermite
functions that we will often use are:
\begin{itemize}
\item $h_{n}\in \mathcal{S}$ for all $n \in \mathbb{N}_{0}$,

\item $h_n$ is an even (odd) function if $n$ is even (odd),

\item $\sqrt{n+1}h_{n+1}(x)+2h_n^{\prime}(x)=\sqrt{n} h_{n-1}(x)$
for all $n \in \mathbb{N}_{0}$,

\item $\sqrt{n+1}h_{n+1}(x)=xh_n(x)-\sqrt{n} h_{n-1}(x)$ for all
$n \in \mathbb{N}_{0}$,

\item $\{h_{n}: n \in \mathbb{N}_{0} \}$ is an orthonormal basis
of $L^{2}(\mathbb{R})$,

\item $(N+1)h_{n}=(n+1)h_{n}$ for all $n \in \mathbb{N}_{0}$.
\end{itemize}

\noindent  From  the two last properties   we have

\[
\|\varphi\|^2_{m}=\sum_{n=0}^{\infty}(n+1)^{2m}<\varphi,h_n>^2,
\]
where $<\varphi,h_n>=\int \varphi(x)h_n(x)dx$ are the
Fourier-Hermite coefficients of the expansion of $\varphi$.

\noindent The Hermite representation theorem for $\mathcal{S}$
($\mathcal{S}^{\prime}$) states a topological isomorphism from
$\mathcal{S}$ ($\mathcal{S}^{\prime}$) onto the space of sequences
$\mathbf{s}$ ($\mathbf{s}^{\prime}$).

\noindent Let $\mathbf{s}$ be the space of rapidly decreasing
sequences
\[
\mathbf{s}=\{(a_n)\in \ell^2:\sum_{n=0}^{\infty}(n+1)^{2m}\mid
a_n\mid^2<\infty, \; \mbox{for all } \; m \in \mathbb{N}_0\}.
\]
The space $\mathbf{s}$ is a locally convex space with the sequence
of norms
\[
\| (a_n)\|_m = (\sum_{n=0}^{\infty}(n+1)^{2m}\mid
a_n\mid^2)^{\frac{1}{2}}
\]
or with the equivalent sequence of norms
\[
\mid(a_{n})\mid_{m,\infty} =  \sup_{n} (n+1)^{m}| a_n|.
\]
The topological dual space to $\mathbf{s}$, denoted by
$\mathbf{s}^{\prime}$, is given by
\[
\mathbf{s}^{\prime}=\{(b_n): \mbox{for some
}\;(C,m)\in\mathbb{R}\times \mathbb{N}_0, \; \mid b_n \mid \leq
C(n+1)^m \mbox{ for all } n \}.
\]
The natural pairing of elements from $\mathbf{s}$ and
$\mathbf{s}^{\prime}$, denoted by $\langle \cdot, \cdot \rangle$,
is given by
\[
\langle (b_n), (a_n) \rangle = \sum_{n=0}^{\infty}b_na_n
\]
for $(b_n) \in \mathbf{s}^{\prime}$ and $(a_n) \in \mathbf{s}$.

It is clear that $\mathbf{s}^{\prime}$ is an algebra with the
pointwise operations:
\begin{eqnarray*}
(b_n)+(b^{\prime}_n)& = & (b_n+b^{\prime}_n) \\
(b_n)\cdot(b^{\prime}_n)& = & (b_nb^{\prime}_n),
\end{eqnarray*}
and $\mathbf{s}$ is an ideal of $\mathbf{s}^{\prime}$.

The relation between $\mathbf{s}$ ($\mathbf{s}^{\prime}$) and
$\mathcal{S}$ ($\mathcal{S}^{\prime}$) is induced by the Hermite
functions, via Hermite coefficients (evaluation). The following
representation theorem is fundamental in our work, for the proof
see \cite{Red1} pp. 143.

\begin{theorem}[N-representation theorem for $\mathcal{S}$ and $\mathcal{S}^{\prime}$]\label{rpreS}
$\mathbf{a)}$ Let $\mathbf{h}:\mathcal{S}\rightarrow \mathbf{s}$
be the application
\[
\mathbf{h}(\varphi)=(<\varphi,h_n>).
\]
Then $\mathbf{h}$ is a topological isomorphism. Moreover,
\[
\|\mathbf{h}(\varphi)\|_m=\|\varphi\|_m
\]
for all $\varphi \in \mathcal{S}$.

\noindent $\mathbf{b)}$ Let
$\mathbf{H}:\mathcal{S}^{\prime}\rightarrow \mathbf{s}^{\prime}$
be the application $\mathbf{H}(T)=(T(h_n))$. Then $\mathbf{H}$ is
a topological isomorphism. Moreover, if $T \in
\mathcal{S}^{\prime}$ we have that
\[
T=\sum _{n=0}^{\infty} T(h_n)h_n
\]
in the weak sense and for all $\varphi \in \mathcal{S}$,
\[
T(\varphi)=\langle \mathbf{H}(T),\mathbf{h}(\varphi)\rangle.
\]
\end{theorem}
\noindent We say that the sequences $\mathbf{h}(\varphi)$ and
$\mathbf{H}(T)$ are the {\it Hermite coefficients} of the tempered
function $\varphi$ and the distribution $T$, respectively.

Next we provide the Hermite coefficients for some tempered
distributions.

\subsection{The delta distribution.} (see \cite{Miku} pp 191.)
\begin{equation}\label{delta} \delta (h_{n}) = h_{n}(0)  =
\left \{
\begin{array}{lll}

           \frac{(-1)^{\frac{n}{2}}}{\sqrt[4]{2\pi}}\sqrt{\frac{1}{2}\frac{3}{4}\cdots\frac{n-1}{n}}& & \mbox{for $n$ even},        \\
           0 & & \mbox{for $n$ odd}.
\end{array}
\right .
\end{equation}

\subsection{The constant distribution $\mathbf{1}$.} (see \cite{Miku} pp 190.)

\begin{equation}\label{uno}
1(h_{n})=\int_{-\infty}^{\infty} h_{n}(x) \; dx
  = \left \{
\begin{array}{lll}
          \sqrt[4]{8\pi}\sqrt{\frac{1}{2}\frac{3}{4}\cdots\frac{n-1}{n}} & & \mbox{for $n$
           even},\\
           0 & & \mbox{for $n$ odd.}

\end{array}
\right.
\end{equation}

\subsection{The $x_+^p$ distribution.} (see \cite{Sad1} pp 162.)

 We recall that $<x_+^p,\phi>= \int_{0}^{\infty}x^p \phi(x) \;
 dx$.
\begin{equation}\label{Heaviside}
x_+^p(h_{n}) = \left \{
\begin{array}{ll}
(\sqrt{2\pi}n!)^{- \frac{1}{2}} 2^p \Gamma(
\frac{p+1}{2})W_n(2p+1)
               & \mbox{for $n$ even},\\
   (\sqrt{2\pi}n!)^{-\frac{1}{2}}2^{p+1}\Gamma(\frac{p+2}{2})W_n(2p+1) & \mbox{for $n$ odd}
\end{array}
\right .
\end{equation}
where $W_n(x)$ are polynomials such that $W_0(x)=W_1(x)=1$ and
\[
W_{n+2}(x)=xW_n(x)+n(n-1)W_{n-2}(x).
\]
Note that if $p=0$, then $x_+^p$ is the Heaviside distribution
$H$.

\subsection{The $\delta^{\prime}$ distribution.}

\begin{equation}\label{deltaprima}
\delta^{\prime}(h_n) =-h_n^{\prime}(0)=\sqrt{n}h_{n-1}(0).
\end{equation}

\section{The Tempered  Algebra }

\noindent In order to introduce the tempered algebra we consider
$\mathcal{S}^{\mathbb{N}_0}$ the space of sequences of rapidly
decreasing smooth functions. It is clear that
$\mathcal{S}^{\mathbb{N}_0}$ has the structure of an associative,
commutative differential algebra with the natural operations:
\begin{eqnarray*}
(f_{n})+(g_{n})& = & (f_{n}+ g_{n}) \\
a(f_{n})& = & (af_{n}) \\
(f_{n})\cdot (g_{n})& = & (f_{n} g_{n}) \\
D(f_{n})& = & (Df_{n})
\end{eqnarray*}
where $(f_{n})$ and $(g_{n})$ are in $\mathcal{S}$ and $a \in
\mathbb{R}$.

\begin{definition}\label{defeal} Let
\begin{equation}
\mathcal{H}_{\mathbf{s}^{\prime}}=\{(f_{n})\in
\mathcal{S}^{\mathbb{N}_0} \ : \ \mbox{for each $m \in
\mathbb{N}_0$,} \ (\| f_{n} \|_{m})\in \mathbf{s}^{\prime} \ \}
\end{equation}
and
\begin{equation}
\mathcal{H}_{\mathbf{s}}=\{(f_{n})\in  \mathcal{S}^{\mathbb{N}_0}
\ : \ \mbox{for each $m \in \mathbb{N}_0$,} \  (\| f_{n}
\|_{m})\in \mathbf{s} \ \}.
\end{equation}

\end{definition}

\begin{lemma}\label{lema} $\mathcal{H}_{\mathbf{s}^{\prime}}$ is a subalgebra
of $\mathcal{S}^{\mathbb{N}_0}$ and $\mathcal{H}_{\mathbf{s}}$ is
a differential ideal of $\mathcal{H}_{\mathbf{s}^{\prime}}$.
\end{lemma}
\begin{proof} Let $(f_n),(g_n) \in \mathcal{H}_{\mathbf{s}^{\prime}}$ and $m \in
\mathbb{N}_0$. Applying the inequality (\ref{eq1}), there exists
$r, s \in \mathbb{N}_0$ and a constant $C_m>0$ such that
\[
\| f_n g_n\|_m  \leq  C_m \| f_{n} \|_r\| g_n \|_s.
\]
By definition, there exists constants $D, E>0$ and $p, q\in
\mathbb{N}_0$ such that
\begin{eqnarray*}
\| f_n \|_r & \leq & D(n+1)^p \\
\| g_n \|_s & \leq & E(n+1)^q.
\end{eqnarray*}
Combining these inequalities, we obtain
\[
\| f_n g_n\|_m  \leq C_mDE(n+1)^{p+q}.
\]
This proves that $(\| f_n g_n\|_m ) \in \mathbf{s}^{\prime}$, thus
$(f_n)\cdot(g_n) \in \mathcal{H}_{\mathbf{s}^{\prime}}$.

\noindent Now, we prove that $\mathcal{H}_{\mathbf{s}}$ is an
ideal of $\mathcal{H}_{\mathbf{s}^{\prime}}$. Let $(f_n) \in
\mathcal{H}_{\mathbf{s}^{\prime}}$, $(g_n) \in
\mathcal{H}_{\mathbf{s}}$ and $m \in \mathbb{N}_0$. From
(\ref{defeal}) we have that for each $r \in \mathbb{N}_0$ there
exists a constant $D>0$ and $p\in \mathbb{N}_0$ such that
\[
\| f_n \|_r  \leq  D(n+1)^p
\]
and for all $s,l \in \mathbb{N}_0$,
\[
\|(\| g_n\|_s)\|^2_l =\sum_{n=0}^{\infty}(n+1)^{2l}\| g_n\|^2_s <
\infty.
\]

\noindent Combining the inequality (\ref{eq1}) with the above
equations we obtain
\begin{eqnarray*}
\|(\| f_n g_n\|_m)\|^2_l & = & \sum_{n=0}^{\infty}(n+1)^{2l}\| f_n
g_n\|^2_m \\
& \leq &  C^2_m \sum_{n=0}^{\infty}(n+1)^{2l} \| f_{n} \|^2_r\| g_n \|^2_s \\
 & \leq & C^2_mD^2\sum_{n=0}^{\infty}(n+1)^{2(l+p)}\| g_n \|^2_s \\
 & < & \infty.
\end{eqnarray*}
Note that we have proved that $(\| f_n g_n\|_m) \in \mathbf{s}$,
for all $m \in \mathbb{N}_0$, therefore $(f_n)\cdot(g_n) \in
\mathcal{H}_{\mathbf{s}}$.

\noindent Finally, we prove that if $(f_n) \in
\mathcal{H}_{\mathbf{s}}$ then $(Df_n) \in
\mathcal{H}_{\mathbf{s}}$. In fact, let $m \in \mathbb{N}_0$.
Since $\{ \|\cdot\|_{m}: m \in \mathbb{N}_0\}$ is equivalent to
$\{ \|\cdot\|_{\alpha, \beta, \infty}: \alpha, \beta \in
\mathbb{N}_0\}$ we have that there exists $\alpha, \beta,
m_{\alpha, \beta} \in \mathbb{N}_0$ and a constants
$C_m,C_{\alpha,\beta+1}>0$ such that
\begin{eqnarray*}
\|Df_{n}\|_{m}& \leq & C_m\|Df_{n}\|_{\alpha, \beta, \infty} \\
& = & C_m\|f_{n}\|_{\alpha, \beta+1, \infty} \\
& \leq & C_mC_{\alpha,\beta+1}\|f_{n}\|_{m_{\alpha, \beta}}.
\end{eqnarray*}
As $(f_n) \in \mathcal{H}_{\mathbf{s}}$ we have $(\|Df_{n}\|_{m})
\in \mathbf{s}$, this implies that $(Df_n) \in
\mathcal{H}_{\mathbf{s}}$.
\end{proof}

\begin{proposition}\label{includ} Let $T\in \mathcal{S}^{\prime}$. Then
$(T_{n})\in  \mathcal{H}_{\mathbf{s}^{\prime}}$, where $T_n=\sum
_{j=0}^n T(h_j)h_j$.
\end{proposition}

\begin{proof} From Theorem \ref{rpreS}, there exists a constant $C>0$ and $p\in
\mathbb{N}_0$ such that
\[
| T(h_j) |  \leq  C(j+1)^p
\]
for all $j \in \mathbb{N}_0$. Then

\begin{eqnarray*}
\|T_{n}\|_{m}^{2}& = & \sum_{j=0}^{n}(j+1)^{2m} |T(h_{j})|^2 \\
& \leq & (n+1)^{2m} \sum_{j=0}^{n} |T(h_{j})|^2 \\
& \leq  & C(n+1)^{2(m+p+1)}.
\end{eqnarray*}
This completes the proof.
\end{proof}

\begin{definition}
The tempered algebra is defined as
\[
\mathcal{H}=\mathcal{H}_{\mathbf{s}^{\prime}}/\mathcal{H}_{\mathbf{s}}.
\]
The elements of $\mathcal{H}$ are called tempered generalized
functions.
\end{definition}
Let $(f_n) \in \mathcal{H}_{s^{\prime}}$ we will use $[f_n]$ to
denote the equivalent class $(f_n) + \mathcal{H}_{s}$.
\begin{proposition}\label{includ2}
Let $\iota:\mathcal{S}^{\prime}\rightarrow \mathcal{H}$ be the
application
\[
\iota(T)=[T_n].
\]
Then $\iota$ is a linear embedding. Moreover, we have that

$\mathbf{a)}$ For all $\varphi \in \mathcal{S}$,
\[
\iota(\varphi)=[\varphi].
\]

$\mathbf{b)}$ For all $T \in \mathcal{S}^{\prime}$
\[
\iota(DT)=D\iota(T).
\]
\end{proposition}

\begin{proof} It is clear from the previous Proposition, that $\iota$ is well defined and a linear
application. We claim that $\iota(T)=0$ implies $T=0$. Since
$(T_n) \in \mathcal{H}_{\mathbf{s}}$, we have
\[
\lim_{n \rightarrow \infty}\|T_{n}\|_{m}=0
\]
for all $m \in \mathbb{N}_0$, therefore the sequence $(T_{n})$
converges weakly to $0$, which proves the claim.

\noindent $\mathbf{a)}$ Let $\varphi\in \mathcal{S}$. We have that
$\iota (\varphi)=[\varphi_n]$ where $\varphi_{n}=\sum_{j=0}^{n}
<\varphi,h_j> h_{j}$. Then for all $m, s \in \mathbb{N}_0$ we have
that
\begin{eqnarray*}
\lim_{n \rightarrow
\infty}(n+1)^{2s}\|\varphi-\varphi_{n}\|_{m}^{2}& = & \lim_{n
\rightarrow\infty}(n+1)^{2s}\sum_{j=n+1}^{\infty}(j+1)^{2m} |<\varphi,h_j>|^{2} \\
& \leq & \lim_{n \rightarrow \infty}
\sum_{j=n+1}^{\infty}(j+1)^{2(m+s)} |<\varphi,h_j>|^{2}  \\ & = &
0
\end{eqnarray*}
where the last equality follows from $\|\varphi\|_{m+s} < \infty$.
Therefore, we conclude that $(\|\varphi-\varphi_{n}\|_m) \in
\mathbf{s}$. Since $(\varphi -\varphi_n) \in
\mathcal{H}_{\mathbf{s}}$, it follows that
$\iota(\varphi)=[\varphi]$.

\noindent $\mathbf{b)}$ Let $T \in \mathcal{S}^{\prime}$. Using the
above definitions and properties of Hermite functions it follows that:
\begin{eqnarray*}
DT_n & = & D(\sum_{j=0}^nT(h_j)h_j) \\
& = & \sum_{j=0}^nT(h_j)Dh_j \\
& = &
\sum_{j=0}^nT(h_j)\frac{1}{2}(\sqrt{j}h_{j-1}-\sqrt{j+1}h_{j+1})
\\
& = &
-\sum_{j=0}^n\frac{1}{2}(\sqrt{j}T(h_{j-1})-\sqrt{j+1}T(h_{j+1}))h_j
\\
& = &\sum_{j=0}^nDT(h_j)h_j \\
& = & (DT)_n.
\end{eqnarray*}
Therefore $D\iota(T)=[DT_n]=[(DT)_n]=\iota(DT)$.
\end{proof}

\begin{corollary} Let $\varphi,\psi\in \mathcal{S}$. Then
\[
\iota(\varphi\psi)=\iota(\varphi) \cdot \iota(\psi).
\]
\end{corollary}
\begin{proof}
We first observe that
\[
(\varphi\psi)-(\varphi_{n})\cdot(\psi_{n})=(\varphi)\cdot(\psi-\psi_{n})+(\varphi-\varphi_{n})\cdot(\psi).
\]
Applying Proposition \ref{includ2} and Lemma \ref{lema} we obtain
$(\varphi\psi)-(\varphi_{n})\cdot(\psi_{n}) \in
\mathcal{H}_{\mathbf{s}}$. Therefore
$\iota(\varphi\psi)=\iota(\varphi) \cdot \iota(\psi)$.
\end{proof}

\begin{remark}
Let $\mathcal{O}_{M}$ be the ring of multipliers of $\mathcal{S}$
(see \cite{Red1}). We have a natural multiplication from
$\mathcal{O}_{M}$ by $\mathcal{H}$ into $\mathcal{H}$, defined by
\[
g[f_n]:=[gf_n]
\]
where $g \in \mathcal{O}_{M}$ and $[f_n] \in \mathcal{H}$.

\noindent It is easy to check that the product is well defined and
that $\mathcal{H}$ is a $\mathcal{O}_{M}$-module.
\end{remark}

\noindent Now, we give two examples of tempered generalized
functions.

\begin{example} The distribution $\delta$. We have that $\iota(\delta)=[\delta_n]$, where
$\delta_n = \sum_{j=0}^{n}h_j(0)h_j$. Applying the formula
(\ref{delta}) and the following equality

\[
\sum_{j=0}^{n}h_j(x)h_j(y)=\frac{\sqrt{n+1}}{x-y}\Big(h_{n+1}(x)h_{n}(y)-h_{n+1}(y)h_{n}(x)\Big),
\]

we see that

\begin{equation}\label{delta1} \delta_n(x) =
\left \{
\begin{array}{lll}

           \sqrt{n+1}
\frac{(-1)^{\frac{n}{2}}}{\sqrt[4]{2\pi}}\sqrt{\frac{1}{2}\frac{3}{4}\cdots\frac{n-1}{n}}
 \  \ \frac{h_{n+1}(x)}{x}& & \mbox{for $n$ even},        \\
           \sqrt{n+1}
\frac{(-1)^{\frac{n+3}{2}}}{\sqrt[4]{2\pi}}\sqrt{\frac{1}{2}\frac{3}{4}\cdots\frac{n}{n+1}}
 \  \ \frac{h_{n}(x)}{x} & & \mbox{for $n$ odd}.
\end{array}
\right .
\end{equation}
\end{example}

\begin{example} The element $\delta^{2}$. We have that $\delta^{2}\equiv \iota(\delta)\cdot
\iota(\delta)=[\delta_n^2]$. From (\ref{delta1}) it follows that

\begin{equation}\label{deltaquadrado} \delta^2_n (x)=
\left \{
\begin{array}{lll}

           (n+1)
\frac{1}{\sqrt{2\pi}}(\frac{1}{2}\frac{3}{4}\cdots\frac{n-1}{n})
 \  \ \frac{h_{n+1}^{2}(x)}{x^{2}}& & \mbox{for $n$ even},        \\
           (n+1)
\frac{1}{\sqrt{2\pi}}(\frac{1}{2}\frac{3}{4}\cdots\frac{n}{n+1})
 \  \ \frac{h_{n}^{2}(x)}{x^{2}}& & \mbox{for $n$ odd}.
\end{array}
\right .
\end{equation}

\end{example}

\noindent We introduce the concept of association for tempered
generalized functions.

\begin{definition}\label{ass1} Let $[f_n]$ and $[g_n]$ be tempered generalized functions. We say
that $[f_n]$ and $[g_n]$ are associated, denoted by $ [f_n]
\approx [g_n]$, if for all $\varphi \in \mathcal{S}$
\[
\lim_{n\rightarrow\infty}<f_n-g_n,\varphi>=0.
\]
\end{definition}
\noindent We observe that the relation $\approx$ is well defined.
In fact, if $(l_n)\in \mathcal{H}_{\mathbf{s}}$ and $\varphi \in
\mathcal{S}$ we have that
$\lim_{n\rightarrow\infty}<l_n,\varphi>=0$. It follows immediately
that $\approx$ is an equivalence relation on $\mathcal{H}$.
\begin{proposition}\label{equias}
$\mathbf{a)}$ Let $[f_n], [g_n] \in \mathcal{H}$ such that $ [f_n]
\approx [g_n]$. Then $ D^{\alpha}[f_n] \approx D^{\alpha}[g_n]$
for all $\alpha \in \mathbb{N}$.

$\mathbf{b)}$ Let $[f_n], [g_n] \in \mathcal{H}$ such that $ [f_n]
\approx [g_n]$ and $l \in \mathcal{O}_M$. Then $ l[f_n] \approx
l[g_n]$.

$\mathbf{c)}$ Let $T, S \in \mathcal{S}^{\prime}$ such that
$\iota(T)\approx \iota(S)$. Then $T=S$.
\end{proposition}

\begin{proof}
\noindent $\mathbf{a)}$ By integration by parts and hypothesis,
\begin{eqnarray*}
\lim_{n\rightarrow\infty}<D^{\alpha}f_n -D^{\alpha}g_n, \varphi> &
= & \lim_{n\rightarrow\infty}<f_n -g_n, (-1)^{\alpha}D^{\alpha
}\varphi> \\ & = & 0
\end{eqnarray*}
for all $\varphi\in \mathcal{S}$, that is, $ D^{\alpha}
[f_n]\approx D^{\alpha}[g_n]$.

\noindent $\mathbf{b)}$ Let $\varphi\in \mathcal{S}$. As $l \in
\mathcal{O}_M$ we have $l\varphi \in \mathcal{S}$. By assumption,
\begin{eqnarray*}
\lim_{n\rightarrow\infty}<lf_n -lg_n, \varphi> & = &
\lim_{n\rightarrow\infty} <f_n -g_n, l\varphi> \\
& = & 0.
\end{eqnarray*}
We conclude that $ l[f_n] \approx l[g_n]$.

\noindent $\mathbf{c)}$ By definition, $\lim_{n\rightarrow\infty}
\int (T-S)_n(x) \ h_{k}(x) \ dx =(T-S)(h_k)$ for all $k\in
\mathbf{N}_0$. But $\lim_{n\rightarrow\infty} \int (T-S)_n(x) \
h_{k}(x) \ dx =0$ since $\iota(T)\approx \iota(S)$. Applying the
N-representation theorem we conclude that $T=S$.
\end{proof}

\begin{example}  $x \iota( \delta )\approx 0$. In fact,
\[
\lim_{n\rightarrow\infty}<x\delta_{n}, h_{k}> =(xh_{k})(0)=0.
\]
 \end{example}

\noindent Finally, we are interested in the relation between the
symmetric Hermite product of tempered distributions via Hermite
expansions and the product of tempered generalized functions. The
symmetric Hermite product of tempered distribution was introduced
by C. Shen and M. Sun in \cite{Shen-Sun} based on our previous
work \cite{calo}.
\begin{definition}\label{MulI}
Let $S$ and $T$ be tempered distributions. Suppose that for all
$k\in \mathbb{N}\cup\{0\}$ there exists
\[
c_k =\lim_{n\rightarrow\infty}<T_nS_n ,h_k>
\]
and that $(c_k) \in \mathbf{s}^{\prime}$. The symmetric Hermite
product of $S$ and $T$, denoted by $S \bullet T$, is defined to be
the tempered distribution
\begin{equation}\label{pro3}
\sum_{k=0}^{\infty}c_k h_k.
\end{equation}
\end{definition}
\begin{lemma} $\mathbf{a)}$ The symmetric Hermite
product is commutative and distributive.

\noindent $\mathbf{b)}$ The symmetric Hermite product verifies the
Leibnitz rule: Let $S$ and $T$ be in $\mathcal{S}^{\prime}$, then
\[
D(S \bullet T)=DS \bullet T +S \bullet DT.
\]
\noindent $\mathbf{c)}$ Let $S$ and $T$ be in
$\mathcal{S}^{\prime}$ such that there exists $S \bullet T$. Then
\[
\iota(S) \cdot \iota(T) \approx \iota(S \bullet T).
\]
\end{lemma}
\begin{proof}
$\mathbf{a)}$, $\mathbf{b)}$ The proofs are straightforward (see
\cite{calo}).

\noindent $\mathbf{c)}$ It is immediate from the definitions.

\end{proof}

\begin{remark} In order to work with ordinary differential equations
in the tempered generalized  functions setting, we introduce the
algebra $\mathcal{H}_T$ of time dependent tempered generalized
functions. We can proceed in a similar way to the construction of
the algebra $\mathcal{H}$, the details are left to the reader. Let
$\mathcal{S}_T$ be the set of functions $f:[0,T]\times\mathbb{R}
\rightarrow \mathbb{R}$ such that for each $t \in [0,T]$,
$f(t,\cdot) \in \mathcal{S}$ and for each $x \in \mathbb{R}$,
$f(\cdot, x) \in C^1([0,T])$. The set
$\mathcal{H}^T_{\mathbf{s}^{\prime}}$ is given by
\[
\{(f_{n})\in \mathcal{S}_T^{\mathbb{N}_0} \ : \ \mbox{for each $m
\in \mathbb{N}_0$,} \ (\sup_{t \in [0,T]}\| f_{n}(t, \cdot)
\|_{m}),\ (\sup_{t \in [0,T]}\| \frac{\partial f_{n}}{\partial
t}(t, \cdot) \|_{m})\in \mathbf{s}^{\prime} \ \}
\]
and the set $\mathcal{H}^T_{\mathbf{s}}$ is given by
\[
\{(f_{n})\in \mathcal{S}_T^{\mathbb{N}_0} \ : \ \mbox{for each $m
\in \mathbb{N}_0$,} \ (\sup_{t \in [0,T]}\| f_{n}(t, \cdot)
\|_{m}),\ (\sup_{t \in [0,T]}\| \frac{\partial f_{n}}{\partial
t}(t, \cdot) \|_{m})\in \mathbf{s} \  \}.
\]
It is clear that $\mathcal{H}^T_{\mathbf{s}}$ is a differentiable
ideal of the algebra $\mathcal{H}^T_{\mathbf{s}^{\prime}}$. We
define the algebra $\mathcal{H}_T$ as
$\mathcal{H}^T_{\mathbf{s}^{\prime}}/\mathcal{H}^T_{\mathbf{s}}$.
The elements of $\mathcal{H}_T$ are called time depended tempered
generalized functions. It follows immediately that for $[f_n] \in
\mathcal{H}_T$ we have that $\frac{\partial}{\partial t}[f_n(t,
\cdot )]$ is define by $[\frac{\partial}{\partial t}f_n (t, \cdot )]
\in \mathcal{H}$.

\end{remark}
\section{Tempered numbers and point values}

\begin{definition}\label{genseq} The ring of tempered numbers is defined as
\begin{equation}
\mathbf{h}=\mathbf{s}^{\prime}/ \mathbf{s}.
\end{equation}
The elements of $\mathbf{h}$ are called tempered numbers.
\end{definition}
Let $(b_n) \in \mathbf{s}^{\prime}$, we will use $[b_n]$ to denote
the equivalence class $(b_n) + \mathbf{s}$.

\begin{lemma} $\mathbf{a})$ Let $\iota:\mathbb{R}\rightarrow \mathbf{h}$ be the application
\[
\iota(a)=[a].
\]
Then $\iota$ is an embedding.

$\mathbf{b})$ $\mathcal{H}$ is a $\mathbf{h}$-module with the
natural operations.

$\mathbf{c})$ Let $[f_n] \in \mathcal{H}$ and $a \in \mathbb{R}$.
Then $[f_n(a)]\in \mathbf{h}$.
\end{lemma}
\begin{proof}
 $\mathbf{a})$ It is clear that $(a) \in \mathbf{s}^{\prime}$, then
$\iota(a)=[a]$ is well defined. Assuming that $\iota(a)=0$, we
have that $(a) \in \mathbf{s}$. In particular $\lim_{n \rightarrow
\infty}na=0$, it follows that $a=0$.

\noindent $\mathbf{b})$ We have divided the proof into two parts.
We first prove that for $(b_n) \in \mathbf{s}^{\prime}$ and $(f_n)
\in \mathcal{H}_{\mathbf{s}^{\prime}}$ we have $(b_nf_n) \in
\mathcal{H}_{\mathbf{s}^{\prime}}$. In fact, by definition there
exists constants $E, F>0$ and $p, q\in \mathbb{N}_0$ such that
\begin{eqnarray*}
\| f_n \|_m & \leq & E(n+1)^p \\
| b_n | & \leq & F(n+1)^q.
\end{eqnarray*}
Combining these inequalities, we obtain
\[
\| b_nf_n \|_m  \leq EF(n+1)^{p+q}.
\]
This proves that $(\| b_nf_n \|_m ) \in \mathbf{s}^{\prime}$, thus
$(b_nf_n) \in \mathcal{H}_{\mathbf{s}^{\prime}}$.

\noindent Finally, the proof is completed by showing that for
$(a_n) \in \mathbf{s}$ and $(f_n) \in
\mathcal{H}_{\mathbf{s}^{\prime}}$ or $(a_n) \in
\mathbf{s}^{\prime}$ and $(f_n) \in \mathcal{H}_{\mathbf{s}}$ we
have $(a_nf_n) \in \mathcal{H}_{\mathbf{s}}$.

\noindent $\mathbf{c})$ Since $\delta_a \in \mathcal{S}^{\prime}$,
there exists a constant $C>0$ and $m \in \mathbb{N}_0$ such that
\[
|f_n(a)|=|\delta_a(f_n)| \leq C\|f_n \|_{m},
\]
for all $n \in \mathbb{N}_0$.

\noindent Combining the above inequality with $(\|f_n \|_{m})\in
\mathbf{s}^{\prime}$ we conclude that $(f_n(a)) \in
\mathbf{s}^{\prime}$.
\end{proof}

\begin{remark}
We observe that $\mathbf{h}$ is not a field, since there exist
zero divisors in  $\mathbf{h}$. In fact, $[1+(-1)^n],
[1+(-1)^{n+1}] \in \mathbf{h}$ are non zero and its product is
zero.
\end{remark}
\begin{definition} The point value of $[f_n]\in \mathcal{H}$ in $a\in
\mathbb{R}$, denoted by $[f_n](a)$, is defined to be $[f_n(a)]$.
\end{definition}
\begin{example} The point value of $\delta$ in $a\in
\mathbb{R}$. From  (\ref{delta1}) we have that
$\iota(\delta)(a)=[a_n]$, where

\begin{equation}\nonumber a_n=
\left \{
\begin{array}{lll}

           \frac{\sqrt{n+1}}{a}h_{n+1}(a)h_{n}(0)& & \mbox{for $n$ even},        \\
           -\frac{\sqrt{n+1}}{a}h_{n}(a)h_{n+1}(0)& & \mbox{for $n$ odd}.
\end{array}
\right .
\end{equation}
\end{example}

\begin{example} The point value of $x_{+}$ in 0. It is easy to check that
$\iota(x_+)(0)=[a_n]$, where

\begin{equation}\nonumber a_n=
\left \{
\begin{array}{lll}

           \sqrt{n+1}h_{n}(0) \int_{0}^{\infty} h_{n+1}(x) \ dx& & \mbox{for $n$ even},        \\
           \sqrt{n}h_{n-1}(0) \int_{0}^{\infty} h_{n}(x) \ dx& & \mbox{for $n$ odd}.
\end{array}
\right .
\end{equation}
\end{example}

\noindent  We introduce the concept of association for tempered
numbers.
\begin{definition}\label{ass2}
Let $[a_n]$ and $[b_n]$ be tempered numbers. We say that $[a_n]$
and $[b_n]$ are associated, denoted by $ [a_n] \approx [b_n]$, if
\[
\lim_{n\rightarrow\infty}(a_n-b_n)=0.
\]
\end{definition}
\noindent We observe that the relation $\approx$ is well defined
and that $\approx$ is an equivalence relation on $\mathbf{h}$.
\begin{example} $\iota(x_+)(0) \approx 0$.
\end{example}

\section{Integration and Fourier transform}

In this section we present the integration theory of tempered
generalized functions and the Fourier transform.
\begin{definition}\label{defin} Let $[f_n]\in
\mathcal{H}$ and $A$ be a Lebesgue measurable set. The integral of
$[f_n]$ on $A$, denoted by $\int_A [f_n](x) \ dx$ is defined to be
\begin{equation}\label{integr}
[\int_Af_{n}(x) \ dx].
\end{equation}
\end{definition}

\noindent We observe that the integral is well defined as an
element of $\mathbf{h}$. In fact, as $1_A$ is a tempered
distribution there exists a constant $C>0$ and $m \in
\mathbb{N}_0$ such that
\[
|\int_{A} g(x)\ dx|\leq C \|g \|_{m},
\]
for all $g \in \mathcal{S}$. In particular, for $[f_n]\in
\mathcal{H}$ we have that
\[
|\int_{A} f_{n}(x)\ dx|\leq C \|f_{n} \|_{m}.
\]
As $(\|f_{n} \|_{m}) \in \mathbf{s}^{\prime}$, we conclude that
$\int_A [f_n](x) \ dx \in \mathbf{h}$.

In the next Lemma we collect some fundamental properties of the
integral of tempered generalized functions.

\begin{lemma} Let $[f_n]$ and $[g_n]$ be tempered
generalized functions, $a=[a_n] \in \mathbf{h}$ and $\alpha\in
\mathbb{N}$. Then

\noindent $\mathbf{a)}$ Let $A$ and $B$ disjoint Lebesgue
measurable sets. Then
\[
\int_{A\cup B} [f_n](x) \ dx= \int_{A} [f_n](x) \ dx + \ \int_{B}
[f_n](x) \ dx.
\]
$\mathbf{b)}$
\[
\int_{A} ([f_n]+a[g_n])(x) \ dx= \int_{A} [f_n](x) \ dx + \ a
\int_{A} [g_n](x) \ dx.
\]

\noindent $\mathbf{c)}$ Let $\varphi\in \mathcal{S}$. Then
\[
\iota(\int_{A} \varphi \ dx)= \int_{A} \iota(\varphi) \ dx.
\]

\noindent $\mathbf{d)}$ "Rule of integration by parts"
\[
\int_{\mathbb{R}}  [f_n] D^{\alpha} [g_n](x) \ dx=(-1)^{\alpha}
\int_{\mathbb{R}}( D^{\alpha}[f_n]) [g_n](x)\ dx.
\]

\noindent $\mathbf{e)}$ Let $\varphi\in \mathcal{S}$ and $T \in
\mathcal{S}^{\prime}$. Then
\[
\int_{\mathbb{R}} \iota(T) \cdot \iota (\varphi)(x) \ dx =
\iota(T(\varphi)).
\]
\end{lemma}
\begin{proof}
The proof of $\mathbf{a)}$, $\mathbf{b)}$ and $\mathbf{c)}$ are
immediate.

\noindent $\mathbf{d)}$ We have that
\begin{eqnarray*}
\int_{\mathbb{R}}  [f_n] D^{\alpha} [g_n](x) \ dx & = &
[\int_{\mathbb{R}}  f_n D^{\alpha} g_n(x) \ dx] \\
& = & [\int_{\mathbb{R}} (-1)^{\alpha} D^{\alpha}f_n (x) g_n(x) \
dx] \\ & = & (-1)^{\alpha} \int_{\mathbb{R}}( D^{\alpha}[f_n])
[g_n](x)\ dx.
\end{eqnarray*}
\noindent $\mathbf{e)}$ By definitions
\begin{eqnarray*}
\int_{\mathbb{R}} \iota(T) \cdot \iota (\varphi)(x) \ dx & = &
[\int_{\mathbb{R}} T_n(x) \varphi(x) \ dx]
\end{eqnarray*}
Let us prove that $(T(\varphi)-\int T_n(x) \varphi(x) \ dx)\in
\mathbf{s}$. Combining definitions, the  N-representation theorem
 and $\lim_{n \rightarrow \infty}<\varphi,h_j>=0$,
we obtain that there exists $n_0 \in \mathbb{N}$ such that if $n
\geq n_0$,
\begin{eqnarray*}
|T( \varphi)-\int_{\mathbb{R}} T_{n}(x) \varphi(x) \ dx| & = &
|T-T_{n}(
\varphi)| \\
& \leq & \sum_{j=n+1}^{\infty} |T(h_j)| |<\varphi,h_j>| \\
& \leq & \sum_{j=n+1}^{\infty} C(j+1)^p |<\varphi,h_j>| \\
& \leq & \sum_{j=n+1}^{\infty} C(j+1)^{2p} |<\varphi,h_j>|^2.
\end{eqnarray*}
Since $\varphi \in \mathcal{S}$, it follows that
\[
\lim_{n \rightarrow \infty} \sum_{j=n+1}^{\infty} (j+1)^{q}
|<\varphi,h_j>|^2=0
\]
for all $q \in \mathbb{N}_0$. Combining the above inequalities we
see that
\[
\lim_{n \rightarrow \infty}(n+1)^{r}|T( \varphi)-\int_{\mathbb{R}}
T_{n}(x) \varphi(x) \ dx|=0,
\]
for all $r \in \mathbb{N}_0$. This proves that $(T( \varphi)-\int
T_n(x) \varphi(x) \ dx)\in \mathbf{s}$, which completes the proof.
\end{proof}

\begin{example} Let $T\in \mathcal{S}^{\prime}$. Then
\begin{eqnarray*}
\int_{\mathbb{R}} \iota (T)(x) \ dx & = & [\int_{\mathbb{R}}
T_n(x) \ dx] \\
& = & [1( T_n)] \\
& = & [\sum_{ j \ even }^{n} T(h_j)h_j(0)(-1)^{\frac{j}{2}}],
\end{eqnarray*}
where the last equality follows from formula (\ref{uno}).
\end{example}

\begin{example} $\delta^{2}$. From formula (\ref{deltaquadrado}) we have that
\begin{eqnarray*}
\int_{\mathbb{R}} \delta^{2}(x) \ dx & = & [\int_{\mathbb{R}}
\delta_n^{2}(x) \ dx] \\
& = & [ \frac{(n+1)}{\sqrt{2\pi}}
\frac{1}{2}\frac{3}{4}\cdots\frac{n-1}{n}].
\end{eqnarray*}
\end{example}
\noindent The Fourier transform and the convolution are very
important tools of classical and modern analysis, our aim is
introduce these operations in the context of tempered generalized
functions. We recall that the Fourier transform
$\mathcal{F}:\mathcal{S}\rightarrow\mathcal{S}$ is defined by
\[
\mathcal{F}(\varphi)(t)=\frac{1}{\sqrt{2\pi}}\int_{\mathbb{R}}e^{-itx}\varphi(x)
\; dx
\]
and the convolution $\ast: \mathcal{S} \times \mathcal{S}
\rightarrow\mathcal{S}$ is defined by
\[
\varphi \ast \psi (t) = \int_{\mathbb{R}}\varphi(t-x)\psi(x) \;
dx.
\]
For a fuller treatment of the issues discussed below,  we refer the reader to
\cite{Red2}.
\begin{definition} The Fourier transform of a generalized tempered function $[f_n]$, denoted by $\mathcal{F}([f_n])$,
is defined to be $[\mathcal{F}(f_n)]$.
\end{definition}
\noindent We observe that the above definition is independent of
the representatives, because for all $m \in \mathbb{N}_0$ and
$\varphi \in \mathcal{S}$ we have that $\|\mathcal{F}(\varphi)\|_m
= \| \varphi \|_m$.

\noindent Here are some elementary properties of the Fourier
transform and convolution.

\begin{theorem}
$\mathbf{a)}$ The Fourier transform $\mathcal{F}: \mathcal{H}
\rightarrow\mathcal{H}$ is a linear isomorphism and its inverse is
given by
\[
\mathcal{F}^{-1}([f_n])=[\mathcal{F}^{-1}(f_n)].
\]

\noindent $\mathbf{b)}$ Let $[f_n] \in \mathcal{H}$ and $\alpha
\in \mathbb{N}_0$. Then
\begin{eqnarray*}
\mathcal{F}(D^{\alpha}[f_n]) & = & (ix)^{\alpha}\mathcal{F}([f_n])
\\
\mathcal{F}(x^{\alpha}[f_n]) & = & i^{\alpha}D^{\alpha}
\mathcal{F}([f_n]).
\end{eqnarray*}

\noindent $\mathbf{c)}$ Let $T$ be a tempered distribution. Then
$\iota(\mathcal{F}(T))= \mathcal{F}(\iota (T))$.

\end{theorem}
\begin{proof}
\noindent $\mathbf{a)}$ Define $\mathcal{G}: \mathcal{H}
\rightarrow \mathcal{H}$ by
$\mathcal{G}([f_n])=[\mathcal{F}^{-1}(f_n)]$. We observe that
$\mathcal{G}$ is well defined, because for any $m \in
\mathbb{N}_0$ and $\varphi \in \mathcal{S}$ we have $\|\varphi\|_m
=\| \mathcal{F}^{-1} \varphi\|_m$.  It is clear that
$\mathcal{F}\circ \mathcal{G} = I_{\mathcal{H}}$ and
$\mathcal{G}\circ \mathcal{F} = I_{\mathcal{H}}$.

\noindent $\mathbf{b)}$ and  $\mathbf{c)}$. The proofs follows
from the definitions and properties of the Fourier transform in
$\mathcal{S}$.
\end{proof}
\begin{definition} Let $[f_n]$ and $[g_n]$ be  tempered generalized  functions. The convolution of $[f_n]$ and $[g_n]$,
denoted by $[f_n]\ast[g_n]$, is defined to be $\mathcal{F}^{-1}(
\sqrt{2 \pi} \mathcal{F}([f_n]) \cdot \mathcal{F}([g_n]))$.
\end{definition}
\begin{theorem}
\noindent $\mathbf{a)}$ Let $[f_n], [g_n] \in \mathcal{H}$. Then
\[
[f_n]\ast[g_n]=[f_n \ast g_n].
\]
\noindent $\mathbf{b)}$ Let $[f_n], [g_n], [h_n] \in \mathcal{H}$
and $\alpha \in \mathbb{N}_0$. Then
\begin{eqnarray*}
[f_n]\ast[g_n] & = & [g_n]\ast[f_n] \\
D^{\alpha}([f_n]\ast[g_n]) & = & (D^{\alpha}[f_n]) \ast[g_n] \\
([f_n]\ast[g_n]) \ast [h_n] & = & [f_n] \ast ([g_n] \ast [h_n])
\end{eqnarray*}

\noindent $\mathbf{c)}$  Let $[f_n], [g_n] \in \mathcal{H}$. Then
\begin{eqnarray*}
\mathcal{F}([f_n]\cdot[g_n]) & = & \frac{1}{\sqrt{2 \pi}}
\mathcal{F}([f_n]) \ast \mathcal{F}([g_n]) \\
\mathcal{F}([f_n]\ast[g_n]) & = & \sqrt{2 \pi} \mathcal{F}([f_n])
\cdot \mathcal{F}([g_n])
\end{eqnarray*}

\noindent $\mathbf{d)}$ Let $T\in \mathcal{S}^{\prime}$ and
$\varphi \in \mathcal{S}$. Then
\[
\iota(T) \ast\iota(\varphi)\approx\iota(T\ast \varphi).
\]
\end{theorem}
\begin{proof}
\noindent $\mathbf{a)}$ The proof is a consequence of the above
theorem and definitions.

\noindent $\mathbf{b)}$ and  $\mathbf{c)}$ The proofs follows from
the definitions and properties of the convolution in
$\mathcal{S}$.

\noindent $\mathbf{d)}$  We have that for all $\psi \in
\mathcal{S}$,
\begin{eqnarray*}
\lim_{n\rightarrow \infty}< (T_{n}\ast\varphi), \psi >
& = & T\ast\varphi(\psi) \\
& = & \lim_{n\rightarrow \infty}< (T \ast \varphi)_n , \psi>.
\end{eqnarray*}
This shows that $\iota(T) \ast\iota(\varphi)\approx\iota(T\ast
\varphi)$.
\end{proof}
\begin{example} The Fourier  transform of $\delta$. By formula (\ref{delta}) we
have that
\[
\mathcal{F}(\iota (\delta))=[\sum_{k=0}^{n} (-i)^{k} h_{k}(0)
h_{k}(x)].
\]
\end{example}

\section{Generalized Stochastic Calculus}

\noindent Let $(\Omega, \mathcal{F}, \{ \mathcal{F}_t: t \in [0,T]
\}, \mathbb{P})$ be a filtered probability space, which satisfies
the usual hypotheses. For a recent account of stochastic calculus
and notations we refer the reader to the book of  Ph. Protter
\cite{Protter}.

\begin{definition}
Let $[f_n] \in \mathcal{H}$, $X$ be a continuous jointly
measurable process and $V$ be a finite variation process. We
define the integral of $[f_n](X)$ in relation to $V$ from 0 to
$t$, denoted by $\int_0^t[f_n](X_s)dV_s$, to be
\[
[\int_0^tf_n(X_s)dV_s].
\]
\end{definition}
\noindent It is clear that for each $\omega \in \Omega$ and $t \in
[0,T]$ we have that $[\int_0^tf_n(X_s)dV_s(\omega)] \in
\mathbf{h}$, because
\[
|\int_0^tf_n(X_s)dV_s(\omega)| \leq \sup_x|f_n(x)||V|_t(\omega)
\]
where $|V|_t(\omega)$ is the total variation of $V$ in $[0,t]$.
\begin{definition}
$\mathbf{a})$ Let $[f_n] \in \mathcal{H}$ and $X$ be a random
variable. We define the expectation of $[f_n](X)$, denoted by
$\mathbb{E}([f_n](X))$, to be $[\mathbb{E}(f_n(X))]$.

$\mathbf{b})$ Let $[f_n] \in \mathcal{H}$, $X$ be a continuous
jointly measurable process and $V$ be a finite variation process
such that $|V|_t$ is integrable. We define the expectation of
$\int_0^t[f_n](X_s)dV_s$, denoted by
$\mathbb{E}(\int_0^t[f_n](X_s)dV_s)$, to be $[
\mathbb{E}(\int_0^tf_n(X_s)dV_s)]$.
\end{definition}
It is easy to check that the above notion is  well
defined. We also observe that the natural definition of expectation
doesn't work. In fact, let $Y_n: \Omega \rightarrow \mathbb{R}$ be
random variables such that $(Y_n(\omega)) \in \mathbf{s}^{\prime}$
for all $\omega \in \Omega$. We have that $[(\mathbb{E}(Y_n))]$ is
dependent of the representatives $(Y_n)$, because if $\{ A_n: n
\in \mathbb{N}\}$ is a partition measurable of $\Omega$ such that
$\mathbb{P}(A_n)=\frac{1}{2^n}$ for all $n \in \mathbb{N}$ and
$(b_n) \in \mathbf{s}^{\prime}$,
$[Y_n(\omega)]=[Y_n(\omega)+b_n2^n1_{A_n}(\omega)]$ and
$(\mathbb{E}(b_n2^n1_{A_n})) =(b_n) \notin \mathbf{s}$.

We can now prove the It\^o formula for tempered generalized
functions. Clearly, this formula is an extension of the classical
It\^o formula via infinite dimensional methods.
\begin{theorem}
Let $[f_n] \in \mathcal{H}$ and $X$ be a continuous
semimartingale. Then
\begin{equation}\label{ito}
[f_n](X_t)=[f_n](X_0)+\int_0^tD[f_n](X_s)dX_s +
\frac{1}{2}\int_0^tD^2[f_n](X_s)d<X>_s
\end{equation}
where $\int_0^tD[f_n](X_s)dX_s(\omega)$ defined by
$[\int_0^tDf_n(X_s)dX_s(\omega)]$ is the It\^o integral of
$[f_n](X)$ in relation to $X$ from 0 to $t$.
\end{theorem}
\begin{proof}
We first show that $[\int_0^tDf_n(X_s)dX_s(\omega)]$ is well
defined. In fact, let $(g_n) \in \mathcal{H}_{\mathbf{s}}$. Since
$(Dg_n) \in \mathcal{H}_{\mathbf{s}}$, we have
$(\int_0^tDg_n(X_s)d<X>_s(\omega)) \in \mathbf{s}$. We see that
$(g_n(X_t(\omega)))$ and $(g_n(X_0(\omega)))$ are in $\mathbf{s}$,
which is clear from the definition of point value. Combining these
facts with the It\^o formula we have
\[
(\int_0^tDg_n(X_s)dX_s(\omega))=(g_n(X_t(\omega)))-(g_n(X_0(\omega)))-
(\frac{1}{2}\int_0^tD^2g_n(X_s)d<X>_s(\omega))
\]
are in $\mathbf{s}$ . Finally we see that
$[\int_0^tDf_n(X_s)dX_s(\omega)] \in \mathbf{h}$ and the formula
(\ref{ito}) holds, this is clear from the It\^o formula applied to
$f_n$,
\[
f_n((X_t(\omega)))=
f_n((X_0(\omega)))+\int_0^tDf_n(X_s)dX_s(\omega)+\frac{1}{2}\int_0^tD^2f_n(X_s)d<X>_s(\omega).
\]
\end{proof}
\begin{remark} Let $f \in \mathcal{S}$. By Proposition \ref{includ2},
$[f]=[f_n]$. Then it is clear that
\[
[ f_n](X_t)  =  [f(X_t)]
\]
and
\[
\int_0^tD^2[f_n](X_s)d<X>_s =  [\int_0^tD^2f(X_s)d<X>_s].
\]
Consequently,
\[
\int_0^tD[f_n](X_s)dX_s = [\int_0^tDf(X_s)dX_s].
\]
In particular, the members of the It\^o formula for $f$ as
function are the same that the members of the It\^o formula for
$f$ as tempered generalized function.
\end{remark}
\begin{remark}
We observe that the members of the It\^o formula for $C^4$ functions
with appropriated decreasing at infinite are associated with the
corresponding members of the It\^o formula as tempered generalized
functions. In fact, we have that $(f_n)$ converges uniformly over
compacts whenever $f$ is twice continuously differentiable and
$O(e^{-cx^2})$ for some $c>1$ as $x \rightarrow \infty$ (see
\cite{stone} for more details). In particular,
\[
[f_n](x) \approx [f(x)]
\]
for all $x \in \mathbb{R}$. Thus
\[
[ f_n](X_t)  \approx  [f(X_t)]~~\mbox{ and }~~[ f_n](X_0)  \approx
[f(X_0)].
\]
If $D^2f \in \mathcal{C}^2$ and $D^2f$ is $O(e^{-cx^2})$ for some
$c>1$ as $x \rightarrow \infty$ we have that
\[
\int_0^tD^2[f_n](X_s)d<X>_s \approx [\int_0^tD^2f(X_s)d<X>_s].
\]
Combining the above identities with the classical It\^o formula
for $f$ we conclude that
\[
\int_0^tD[f_n](X_s)dX_s \approx [\int_0^tDf(X_s)dX_s].
\]
\end{remark}
\begin{remark}
We recall that S. Ustunel \cite{ustunel} and B. Rajeev \cite{rajeev}
obtained It\^o formula for tempered distributions, in the context of
stochastic integration in Hilbert spaces. Our formula (\ref{ito}) is
different from the It\^o formulas given in \cite{rajeev} and
\cite{ustunel}, in particular, the formulas in these references do
not extend the classical It\^o formula as ours does.

\end{remark}
\begin{corollary}
Let $X$ be a semimartingale. Then
\begin{equation}\label{itotanaka}
|X_t-a|=
|X_0-a|+\int_0^t\mathrm{sgn}(X_s-a)dX_s+\int_0^t\delta_a(X_s)d<X>_s.
\end{equation}
\end{corollary}
\begin{proof}
Applying the It\^o formula (\ref{ito}) to the tempered
distribution $|\cdot -a|$ we obtain (\ref{itotanaka}).

\end{proof}

\begin{corollary}
Let $[f_n] \in \mathcal{H}$ and $B$ be a Brownian motion such that
$B_0=0$. Then
\begin{equation}\label{dynkin}
\mathbb{E}([f_n](B_t+x))=[f_n](x)+
\frac{1}{2}\int_0^t\mathbb{E}(D^2[f_n](B_s+x))ds
\end{equation}
\end{corollary}
\begin{proof}
We have
\[
\mathbb{E}([\int_0^tD[f_n](B_s+x)dB_s])=[\mathbb{E}\int_0^tDf_n(B_s+x)dB_s]=0,
\]
because $\int Df_n(B_s+x)dB_s$ is a martingale.
\end{proof}
\begin{corollary}
Let $[f_n] \in \mathcal{H}$. Then
$g_t=[\mathbb{E}(f_n(B_t+\cdot))] \in \mathcal{H}_T$ solves the
Cauchy problem
\begin{eqnarray*}
D_tg & = & \frac{1}{2}D_x^2g \\
g_0 & = & [f_n].
\end{eqnarray*}
\end{corollary}
\begin{proof}
We observe that
\[
\mathbb{E}(f_n(B_t+x))=\int_{\mathbb{R}}f_n(y)p_t(x-y)~dy=f_n \ast
p_t(x)
\]
where $p_t(y)= \frac{1}{\sqrt{2 \pi t}}e^{-\frac{y^2}{2t}}$ is the
heat kernel. As $\ast$ is a continuous operation in $\mathcal{S}$
and $\lim_{n \rightarrow \infty}f_n \ast p_t=f_n$ in $\mathcal{S}$
we conclude that $[g_n] \in \mathcal{H}_T$ where the functions
$g_n:[0,T]\times\mathbb{R} \rightarrow \mathbb{R}$ are given by
$g_n(t,x)=\mathbb{E}(f_n(B_t+x))$. Applying the formula
(\ref{dynkin}) we completes the proof.
\end{proof}
\begin{theorem}
Let $[f_n] \in \mathcal{H}_T$ and $X$ be a continuous
semimartingale. Then
\begin{eqnarray*}
[f_n](X_t) & = &
[f_n](X_0)+\int_0^tD_t[f_n](s,X_s)ds+\int_0^tD_x[f_n](s,X_s)dX_s
\\& & + \frac{1}{2}\int_0^tD_x^2[f_n](s,X_s)d<X>_s.
\end{eqnarray*}
\end{theorem}
\begin{proof}
We observe that $\int_0^tD_t[f_n](s,X_s)ds$ is well defined and
proceed analogously to the proof of the extension of It\^o
formula.
\end{proof}

\end{document}